\documentclass[11pt,psamsfonts]{amsart}

\usepackage{fullpage}
\usepackage{ulem}

\usepackage{amssymb}
\usepackage{amsfonts}
\usepackage{amsmath}
\usepackage{mathrsfs} 
\usepackage{mathtools}
\usepackage{xcolor}

\usepackage{enumerate}
\usepackage[shortlabels]{enumitem}

\usepackage{hyperref} 
\usepackage[noabbrev,capitalise]{cleveref} 
\usepackage{url}

\newtheorem*{corollary*}{Corollary}
\newtheorem*{theorem*}{Theorem}

\newtheorem{theorem}{Theorem}[section]
\newtheorem{corollary}[theorem]{Corollary}
\newtheorem{proposition}[theorem]{Proposition}
\newtheorem{lemma}[theorem]{Lemma}

\newtheorem{hypothesis}[theorem]{Hypothesis}

\theoremstyle{definition}

\theoremstyle{remark}

\newtheorem*{remark}{Remark}
\newtheorem*{remarks}{Remarks}
\numberwithin{equation}{section}

\newcommand{\declarecommand}[1]{\providecommand{#1}{}\renewcommand{#1}}
\newcommand\numberthis{\addtocounter{equation}{1}\tag{\theequation}}
\declarecommand{\R}{\mathbb{R}}
\declarecommand{\Q}{\mathbb{Q}}
\declarecommand{\Z}{\mathbb{Z}}
\declarecommand{\N}{\mathbb{N}}
\declarecommand{\C}{\mathbb{C}}
\declarecommand{\emptyset}{\varnothing}

\declarecommand{\Re}{\mathrm{Re}}
\declarecommand{\Im}{\mathrm{Im}}

\declarecommand{\sL}{\mathscr{L}}
\declarecommand{\epsilon}{\varepsilon}
\declarecommand{\ds}{\displaystyle}

\declarecommand{\edit}[1]{{\color{red}#1}}
\declarecommand{\emph}[1]{\textit{#1}}
\declarecommand{\edit}[1]{{#1}}

\title{Explicit Deuring--Heilbronn phenomenon for Dirichlet $L$-functions}

\author{K\"{u}bra Benl\.{i}} 
\address{
K\"{u}bra Benl\.{i}\\
Department of Mathematics \\
Bo\u{g}az\.{i}\c{c}\.{i} University   \\ 
Bebek, \.{I}stanbul 34342 \\ 
Turkey
(\href{mailto:kubra.benli@boun.edu.tr}{\tt kubra.benli@pt.bogazici.edu.tr})
}

\author{Shivani Goel}
\address{
Shivani Goel\\
Chennai Mathematical Institute\\
H-1 SIPCOT IT Park, Siruseri, Kelambakkam\\ Tamil Nadu, 603103\\
India
(\href{mailto:shivanig@cmi.ac.in}{\tt shivanig@cmi.ac.in})
}

\author{Henry Twiss}
\address{
Henry Twiss\\
Department of Mathematics \\
Brown University \\
167 Thayer St, Room 110 \\
Providence, RI 02906 \\
USA
(\href{mailto:henry\_twiss@brown.edu}{\tt henry\_twiss@brown.edu})
}

\author{Asif Zaman}
\address{
Asif Zaman \\
Department of Mathematics \\
University of Toronto   \\ 
40 St. George Street, Room 6290 \\
Toronto, ON M5S 2E4 \\
CANADA 
(\href{mailto:asif.zaman@utoronto.ca}{\tt asif.zaman@utoronto.ca})
}

\keywords{Deuring--Heilbronn phenomenon,  Landau--Siegel zero,  Selberg’s sieve, zero repulsion}
\subjclass[2020]{11M06, 11M20, 11N36}

\begin{document}

\begin{abstract}
Assuming the existence of a Landau--Siegel zero, we establish an explicit Deuring--Heilbronn zero repulsion phenomenon for Dirichlet $L$-functions modulo $q$. Our estimate is uniform in the entire critical strip, and improves over the previous best known explicit estimate \,due to Thorner and Zaman \edit{\cite{thorner_explicit_2024}}. 

\end{abstract}

\maketitle



\section{Introduction}
For any integer $q \geq 1$, define the function
\begin{equation}
\label{eqn:sL}
\sL_q(s) := \prod_{\chi\pmod{q}} L(s,\chi),
\end{equation}
where the product runs over Dirichlet characters $\chi$ modulo $q$. A result of McCurley \cite{mccurley_explicit_1984} (cf. Kadiri \cite{kadiri_explicit_2018} for improvements) states: for $q \geq 3$, the function $\sL_q(s)$ has at most a single zero  in the region
\begin{equation} \label{eqn:ZFR-McCurley}
\Re(s) > 1- \frac{1}{10 \log \max\{ q, q|\Im(s)|, 10\}}
\end{equation}
and, if such a zero exists, then it is a simple real zero $\beta_1 = \beta_1(q)$ associated to a Dirichlet $L$-function $L(s,\chi_{1})$ where $\chi_{1}$ is a real quadratic character modulo $q$. 
This zero is famously referred to as a \emph{Landau--Siegel zero} (or \emph{Siegel zero} or \emph{exceptional zero}) and, according to conjecture, it is not expected to exist. Assuming $\beta_1$ exists, our main result implies the following explicit form of Deuring--Heilbronn zero repulsion phenomenon.  

\begin{corollary}	\label{cor:explicit} Let $T \geq 4$ be real and $q > 400,000$ be an integer. Assume $\sL_q(s)$ given by \eqref{eqn:sL} has a real zero at $s=\beta_1 > 1 - \frac{1}{10 \log q}$. If $\rho = \beta+i\gamma$ is another zero of $\sL_q(s)$ satisfying $\beta > \tfrac{1}{2}$ and $|\gamma| \leq T$, then
\begin{equation} \label{eqn:main}
\beta  < 1 - \frac{\log\Big(\dfrac{c_4}{(1-\beta_1) (c_1 \log q + c_2 \log T + c_3)} \Big)}{c_1 \log q + c_2 \log T + c_3},
\end{equation}
where 
\[
c_1 = 10, \quad c_2 = 1, \quad c_3 = 107, \quad c_4 = \frac{1}{16}.
\]
\end{corollary}
\begin{remarks}
\begin{enumerate}
    \item Platt \cite{platt_numerical_2016}  has verified computationally that a Landau--Siegel zero does not exist for $q \leq 400,000$, so the above statement is known to be vacuous for such $q$. 
    \item  See the end of \cref{sec:Corollaries} for comments on our choice of constants $c_1, c_2, c_3,$ and $c_4$.
\end{enumerate}
\end{remarks}

We also establish an ineffective non-explicit form with better constants. 

\begin{corollary} \label{cor:non-explicit} Fix $\epsilon > 0$. Let $T \geq 4$ be real and $q > 400,000$ be an integer.  Assume $\sL_q(s)$ given by \eqref{eqn:sL} has a real zero at $s=\beta_1 > 1 - \frac{1}{10 \log q}$. If $\rho = \beta+i\gamma$ is another zero of $\sL_q(s)$ satisfying $\beta > \tfrac{1}{2}$ and $|\gamma| \leq T$, then \eqref{eqn:main} holds  with  
\[
c_1 = \frac{4}{3}+\epsilon, \quad c_2 = \frac{2}{3}+\epsilon, \quad c_3 = 0, \quad c_4 = \frac{1}{24},
\]
provided $qT \geq C(\epsilon)$ for some ineffective sufficiently large positive constant $C(\epsilon)$.  
\end{corollary}

Let us provide some historical background for \cref{cor:explicit,cor:non-explicit}. The constants $c_1$ and $c_2$ are the most important in this context, so we shall focus on them as well as the uniformity of $\beta, q,$ and $T$. All results cited in this paragraph are effective unless otherwise specified. Linnik \cite{linnik_least_1944-1} was the first to establish \eqref{eqn:main} with non-explicit constants $c_1$ and $c_2$ for small values of $T$. This theorem was a pivotal innovation to his celebrated result on the least prime in an arithmetic progression \cite{linnik_least_1944}. Jutila \cite{jutila_linniks_1977} proved a strong form of \eqref{eqn:main} for any fixed $c_1 = c_2 > 2$ provided $\beta > 5/6$ and $qT$ is sufficiently large. Building on Jutila's work, Graham \cite{graham_applications_1977} in his PhD thesis established \eqref{eqn:main} for  any fixed $c_1 = c_2 > 3/2$,  provided $T \leq q^{o(1)}$ and $q$ is sufficiently large. The result is ineffective, but he notes that it could be made effective with modifications to the statement and proof. Using sophisticated optimization techniques, Heath-Brown \cite{heath-brown_zero-free_1992} showed \eqref{eqn:main} holds with any fixed $c_1 = c_2 >  11/12$, provided $\beta > 1 - (\log\log\log q)/(3 \log q)$,  $T = 4,$ and $q$ is sufficiently large. Note that \cref{cor:non-explicit} gives the best values for $c_1$ and $c_2$ in all cases except for the narrow region considered by Heath-Brown; Graham \cite{graham_applications_1977} would have obtained the same $c_1$ as in \cref{cor:non-explicit} using currently available subconvexity bounds, but still with the constraint $T \leq q^{o(1)}$. 

Overall, the main constants $c_1$ and $c_2$ have steadily improved with many innovations, but all the aforementioned results have various restrictions on $\beta, q,$ or $T$, and none of them are completely explicit. These restrictions can vary from harmless to serious, depending on the desired application. The first completely explicit version of \eqref{eqn:main} was recently proved by Thorner--Zaman \cite{thorner_explicit_2024}. Their result implies that
\begin{equation} \label{eqn:ThornerZaman}
c_1 = 54.2, \quad c_2 = 16.9, \quad c_3 = 104.7, \quad c_4 = 0.0002,
\end{equation}
are admissible for \eqref{eqn:main}. While they do not have any meaningful restrictions on $q, T,$ or $\beta$, the constants $c_1$ and $c_2$ are notably weaker compared to some of the previously mentioned non-explicit works. This deficiency occurs because Thorner--Zaman apply the power sum method (namely Theorem 2.2 of Kadiri--Ng--Wong \cite{kadiri-ng-wong_least_prime_2019}) which differs greatly from the approaches of Jutila, Graham, and Heath-Brown. The primary aim of this article is to improve this explicit bound. 

Indeed, \cref{cor:explicit} achieves this goal by adapting the more efficient sieve methods of Graham \cite{graham_applications_1977} and hence obtains results of similar  strength. We construct a mollified sum to detect the non-trivial zero $\rho$ by using two kinds of sieves: one by the exceptional character $\chi_{1}$, and one inspired by Selberg's sieve due to Graham. This construction allows us to produce two competing estimates for the mollified sum, one which exploits the existence of the zero $\rho = \beta+i\gamma$ and the other which exploits the existence of the Siegel zero $\beta_1$.   These estimates do not conflict provided $\beta$ is sufficiently repelled by $\beta_1$, which completes the proof.

Two essential inputs for this argument include the knowledge about the growth of Dirichlet $L$-functions on the critical line and the available lower bounds for $1-\beta_1$. Indeed, both corollaries are deduced from our main theorem: a more flexible explicit version of Deuring--Heilbronn phenomenon for Dirichlet $L$-functions. 

\begin{theorem}\label{thm:main}
   Let $q > 400,000$ and $T \geq 4$. Fix $A,B \geq 1$, $0 < \theta \leq \frac{1}{4}$, and $0 < \epsilon \leq \frac{1}{2}$. Assume that every primitive Dirichlet character $\psi \pmod{q_{\psi}}$ satisfies 
    \[
    |L(\tfrac{1}{2}+it,\psi)| \leq  A  \big( q_{\psi}(1+|t|))^{\theta}
    \quad \text{ for } t \in \R.
    \]
     Assume $\sL_q(s)$ given by \eqref{eqn:sL} has a real zero at $s=\beta_1$ satisfying
     \[
     1 - \frac{1}{10 \log q} < \beta_1 < 1 - \frac{B}{q^{\epsilon} (\log q)^2}. 
     \]
     If $\rho = \beta+i\gamma$ is another zero of $\sL_q(s)$ satisfying $\beta > \tfrac{1}{2}$ and $|\gamma| \leq T$, then
\begin{equation} \label{eqn:main-abstract}
\beta  < 1 -  \frac{\log\big(\frac{\theta}{4(1-\beta_1) \log M} \big)}{\log M},
\end{equation}
where 
\begin{equation}\label{eqn:M}
M = K q^{8\theta+2\epsilon}   T^{4\theta} \quad \text{and} \quad  K = \edit{(3 \times 10^{16})} A^{\edit{8}} B^{-2} e^{8(\log q)^{3/4}} \edit{(\log(14A^2 q^3))^{24} (\log q)^4} . 
\end{equation}
\end{theorem}
Roughly speaking, \cref{thm:main} gives an admissible choice of  $c_1,c_2,c_3,c_4$ in \eqref{eqn:main} satisfying
\[
c_1 \to  8\theta+2\epsilon, \qquad c_2 \to 4\theta, \qquad c_3 \to 0, \qquad c_4 \to \theta/4, 
\]
as $qT \to \infty$. 
Our goal in proving \cref{thm:main} was to minimize these ``asymptotic'' values of $c_1$ and $c_2$, because they tend to dictate many applications. For the sake of simplicity, we did not seriously attempt to minimize the quantity $K$ in \eqref{eqn:M}, which primarily impacts the value of $c_3$ in \cref{cor:explicit}. An interested reader may wish to explore such improvements. 

\cref{thm:main} quickly reveals the inputs to both corollaries. \cref{cor:explicit} relies on an explicit convexity estimate ($\theta = \frac{1}{4}$) due to Thorner--Zaman \cite{thorner_explicit_2024}, and an explicit effective Siegel bound ($\epsilon=\frac{1}{2}$) due to Bordignon \cite{bordignon_explicit_2019}. \cref{cor:non-explicit} benefits from the recent subconvex Weyl bound ($\theta > \frac{1}{6}$) spectacularly proved by Petrow and Young \cite{petrow_weyl_2020,petrow_fourth_2023}, and the classic ineffective bound ($\epsilon > 0$) of Siegel \cite{siegel_uber_1935}. Moreover, one can verify that \cref{thm:main} recovers a more uniform version of Graham's ineffective estimate \cite{graham_applications_1977} with $c_1=c_2 > \frac{3}{2}$ in \eqref{eqn:main} by instead applying Burgess' \cite{burgess_character_1963} subconvexity bound ($\theta > \frac{3}{16}$), which was the best available at the time.

The proof of \cref{thm:main} adapts Graham's argument \cite{graham_applications_1977} with two refinements. First, we maintain complete uniformity by carefully tracking the dependencies on $\beta, q,$ and $T$ at every step.   As a result, we remove Graham's constraint of $T \leq q^{o(1)}$. The calculations are fairly delicate since even a loss of $\log(qT)$ can render the final result useless. Second, we prove a stronger bound for the Dirichlet polynomial defined by the Selberg sieve. Unlike before, this estimate is sensitive to the severity of the Landau--Siegel zero; see the remark after \cref{lem:G-estimate} for details.

Finally, we outline the organization of the paper. \cref{sec:Hypotheses} defines hypotheses for subconvexity and Landau--Siegel zeros with general parameters, and describes recent results in terms of such parameters. \cref{sec:Corollaries} quickly deduces \cref{cor:explicit,cor:non-explicit} assuming \cref{thm:main}. \cref{sec:SelbergSieve} prepares a type of Selberg sieve and establishes standard estimates for sums involving Dirichlet characters. \cref{sec:Proof} combines these elements to prove \cref{thm:main}.

\section{Subconvexity and Landau--Siegel zeros}
\label{sec:Hypotheses}

In this section, we formally state the two general hypotheses appearing in \cref{thm:main} and review some of the best available progress towards them.  If explicit improvements are made regarding subconvexity or Landau--Siegel zeros, then \cref{thm:main} immediately yields new estimates with improved constants. Our first hypothesis is an explicit form of subconvexity bound that holds uniformly for Dirichlet $L$-functions.

\begin{hypothesis}\label{hypothesis_A}
Fix constants $A \geq 1$ and $0 < \theta \leq \tfrac{1}{4}$. Assume for $t \in \R$ and any primitive Dirichlet character $\psi \pmod{q_{\psi}}$  that
\[
    |L(\tfrac{1}{2}+it,\psi)| \leq A \big( q_{\psi} (1+|t|) \big)^{\theta}. 
\]
\end{hypothesis}

Below are some of the current best instances of \cref{hypothesis_A}.

\begin{theorem}[Petrow--Young] \label{thm:PY}
For $t \in \R$ and any primitive character $\psi \pmod{q_{\psi}}$,
\[
|L(\tfrac{1}{2}+it,\psi)| \ll_{\epsilon} (q_{\psi}(1+|t|))^{1/6+\epsilon},
\]		
for any $\epsilon > 0$. 
\end{theorem}
\begin{proof}
See Theorem 1.1 of \cite{petrow_fourth_2023}, which improves upon Burgess \cite{burgess_character_1963}. 
\end{proof}

\begin{proposition}[Thorner--Zaman] \label{lem:ExplicitConvexity}
For $t \in \R$ and any primitive character $\psi \pmod{q_{\psi}}$,
	\[
	|L(\tfrac{1}{2}+it,\psi)| \leq 2.97655(q_{\psi} (1+|t|))^{1/4}.
	\]
\end{proposition}
\begin{proof}
    This follows from Proposition 2.10 of \cite{thorner_explicit_2024}, which improves on Hiary \cite{hiary_explicit_2016}.
\end{proof}
\begin{remark}
     Explicit improvements of \cref{lem:ExplicitConvexity} are available in the literature; see, for example, \edit{\cite{hiary_improved_2024,francis_explicit_2022}}  for explicit subconvexity estimates. These can lead to some improvements over \cref{cor:explicit} but require more intricate calculations and a modified \cref{hypothesis_A} with more flexible exponents for $q_{\psi}$ \edit{and} $(1+|t|)$. We opted for simplicity. 
\end{remark}

Assuming \cref{hypothesis_A}, we extend the bound to a larger subregion of the critical strip.

\begin{lemma} \label{lem:Phragmen} Let $q \geq 3$ be an integer. Assume \cref{hypothesis_A} holds with fixed $A \geq 1$ and $0 < \theta \edit{\leq} \tfrac{1}{4}$.      Fix $0 < \eta < 1$. If $\psi \pmod{q_{\psi}}$ is a primitive character,  then  
    \[
    |L(\sigma+it,\psi)| \leq \Big( A  ( 2q_{\psi}(1+|t|))^{\theta} \Big)^{\tfrac{1+\eta-\sigma}{1/2+\eta}} \Big(1+\frac{1}{\eta} \Big)^{\tfrac{\sigma-1/2}{1/2+\eta}}.
    \] 
   for $\frac{1}{2} \leq \sigma \leq 1+\eta$. Additionally, if $\chi \pmod{q}$  is any character, then 
    \[
    |L(\sigma+it,\chi)| \leq \Big( A  ( 2q(1+|t|))^{\theta} \Big)^{\tfrac{1+\eta-\sigma}{1/2+\eta}} \Big(1+\frac{1}{\eta} \Big)^{\tfrac{\sigma-1/2}{1/2+\eta}} \exp\Big((\log q)^{\tfrac{2-\sigma}{2}} \Big)
    \]
    for $\frac{1}{2} \leq \sigma \leq 1+\eta$. 
\end{lemma}
\begin{proof}
For the primitive character $\psi$, \cref{hypothesis_A} and the inequality $(1+|t|) \leq 2|1+it|$ together imply
\[
    |L(\tfrac{1}{2}+it,\psi)|  \leq  A \big( 2q_{\psi}(|1+it|))^{\theta}.
\]
On the other hand, the Dirichlet series representation implies $|L(1+\eta+it)| \leq \zeta(1+\eta) \leq 1 + \eta^{-1}$. The result follows from the Phragmen--Lindel\"{o}f convexity principle. The general case for $\chi$ follows from the primitive case and the additional inequality
\[
    \Big| \prod_{p \mid q} \Big( 1 - \frac{\chi(p)}{p^{\sigma+it}} \Big) \Big| \leq \exp\Big((\log q)^{\frac{2-\sigma}{2}} \Big).
\]
To justify this last estimate, notice that 
\[
    \Big| \prod_{p \mid q} \Big( 1 - \frac{\chi(p)}{p^{\sigma+it}} \Big) \Big| \leq \prod_{p \mid q} \Big( 1 + \frac{1}{p^{\sigma}} \Big)  = \exp\Big( \sum_{p \mid q} \log\Big(1+p^{-\sigma}\Big)  \Big) \leq \exp\Big( \sum_{p \mid q} p^{-\sigma}  \Big),
 \]
where the last equality follows since $\log(1+x) < x$ for $x > 0$. By H\"{o}lder's inequality, 
\begin{align*}
    \sum_{p \mid q} p^{-\sigma} = \sum_{p \mid q}  \frac{ (\log p)^{\frac{2-\sigma}{2}}}{p^{\sigma} (\log p)^{\frac{2-\sigma}{2}} } 
    &\leq  \Big( \sum_{p \mid q} \frac{1}{p^2 (\log p)^{\frac{2-\sigma}{\sigma}} } \Big)^{\frac{\sigma}{2}} \Big( \sum_{p \mid q} \log p \Big)^{\frac{2-\sigma}{2}} \\
    & \leq \Big( \sum_{n=2}^{\infty} \frac{1}{n^2 \log n} \Big)^{\frac{\sigma}{2}} (\log q)^{\frac{2-\sigma}{2}}.
\end{align*}
 The desired estimate follows upon noting $\sum_{n=2}^{\infty} \frac{1}{n^2 \log n} < 1$. 
\end{proof}

Our second hypothesis is an explicit upper bound on the putative Landau--Siegel zero.

\begin{hypothesis}\label{hypothesis_B} 
Fix $B > 0$ and $0 < \epsilon \leq \tfrac{1}{2}$. For any integer $q \geq 3$, if $\sL_q(s)$ given by \eqref{eqn:sL}  has a real zero at $s = \beta_1$ satisfying $\beta_1 > 1-1/(10 \log q)$, then 
\[
    \beta_1 < 1 - \frac{B}{q^{\epsilon}(\log q)^2}.
\]
\end{hypothesis}

We record the classic instance  due to Siegel \cite{siegel_uber_1935} and a recent explicit version due to Bordignon.

\begin{theorem}[Siegel] \label{thm:Siegel}
    \cref{hypothesis_B} holds with any $0 < \epsilon < \frac{1}{2}$ and some ineffective positive constant $B = B(\epsilon) > 0$.
\end{theorem}

\begin{theorem}[Bordignon] \label{thm:Bordignon}
   \cref{hypothesis_B} holds with  $\epsilon = \tfrac{1}{2}$ and $B = 100$. 
\end{theorem}
\begin{proof}
    For $\epsilon = \tfrac{1}{2}$, \cite[Theorem 1.2]{bordignon_explicit_2019} gives $B = 800$ for odd characters and \cite[Theorem 1.1]{BORDIGNON2020481} gives $B = 100$ for even characters. The theorem is weaker than these two results.
\end{proof}

The quantity $1-\beta_1$ is well known to be closely related to the special value $L(1,\chi_1)$. We provide such an explicit estimate similar to the one by Friedlander--Iwaniec \cite{friedlander_note_2018}. 

\begin{lemma} \label{lem:RatioZeroL1}
  Let $q > 400,000$ be an integer. If $1 - \frac{1}{10 \log q} < \beta_1 < 1$ is a real zero of $L(s,\chi_1)$ where $\chi_1 \pmod{q}$ is a real quadratic character, then 
    \[
       0.72 \leq  \frac{L(1,\chi_1)}{1-\beta_1} \leq 0.18 (\log q)^2. 
    \]
\end{lemma}
\begin{proof} The upper bound follows from \cite[Theorem 1.2]{bordignon_explicit_2019}, so it suffices to prove the lower bound.   
	For $n \geq 1$, define $a(n) = \sum_{d \mid n} \chi_1(d)$. For $x \geq 3$, define
    \[
    S(x) = \sum_{n \leq x} a(n) n^{-\beta_1} \Big(1 - \frac{n}{x} \Big),
    \]
    so by Mellin inversion,
    \[
    S(x) = \int_{(2)} \zeta(s+\beta_1) L(s+\beta_1,\chi_1) \frac{x^s}{s(s+1)} ds. 
    \]
    Shifting to $\Re(s) = \frac{1}{2}-\beta_1$ and applying \cref{lem:Phragmen} with \cref{lem:ExplicitConvexity}, it follows that 
    \[
    \Big| S(x) -  \frac{L(1,\chi_1) x^{1-\beta_1}}{(1-\beta_1)(2-\beta_1)}\Big| \leq x^{1/2-\beta_1} q^{1/4} e^{(\log q)^{3/4}} \int_{-\infty}^{\infty} \frac{(2.97655)^2 \cdot 2^{1/2} (1+|t|)^{1/2} }{(( \frac{1}{2}-\beta_1)^2 + t^2 )^{1/2} ( ( \frac{3}{2}-\beta_1)^2 + t^2 )^{1/2}} dt. 
    \]
    Since $q > 400,000$, it follows that $\beta_1 > 1 - \frac{1}{10 \log q} >  0.992$ and $e^{(\log q)^{3/4}} \leq q^{0.53}$. We may therefore crudely bound the righthand side by at most 
    \[
    2 x^{-0.5+\frac{1}{10 \log q}} q^{0.78} \int_{0}^{\infty}  \frac{(2.97655)^2 2^{1/2}  (1+t)^{1/2}}{(0.492+t^2)^{1/2} (0.500 + t^2)^{1/2}} dt < 92.7 x^{-0.5+\frac{1}{10 \log q}} q^{0.78} < 125.2 q^{-0.72},
    \]
    upon choosing $x = q^3$. Rearranging we find that
    \begin{equation}
    0.740 \cdot S(q^3) -  126 q^{-0.72} < \frac{L(1,\chi_1)}{1-\beta_1}, 
    \label{eqn:RatioZeroL1}
    \end{equation}
    as $0.740 < \frac{1}{e^{3/10}} < \frac{2-\beta_1}{q^{3(1-\beta_1)}} < \frac{1.008}{1}$ and $1.008 \times 125.2 < 126$. Since $a(n) \geq 0$ for all integers $n$, we see that $S(q^3) \geq 1 - q^{-3}$ which yields the desired lower bound via  \eqref{eqn:RatioZeroL1} for $q > 400,000    $.   
\end{proof}

\section{Proofs of Corollaries 1.1 and 1.2 assuming Theorem 1.3} \label{sec:Corollaries}

We provide short proofs along with relevant numerics. 

\begin{proof}[Proof of \cref{cor:explicit}]
    After applying \cref{thm:main} with $\theta = \frac{1}{4}$ and $A=3$ from \cref{lem:ExplicitConvexity}, and with $\epsilon = \frac{1}{2}$ and $B = 100$ from \cref{thm:Bordignon}, we obtain $M$ and $K$ defined by \eqref{eqn:M}. Since the righthand side of \eqref{eqn:main} increases with $M$, it remains to check that if  $c_1 = 10, c_2 = 1$, and $c_3= 107$, then 
    \[
    \log M = (8\theta+2\epsilon) \log q + 4\theta \log T + \log K \leq c_1 \log q + c_2 \log T + c_3. 
    \]
   for $q > 400,000$ and $T \geq 4$.  After evaluating all variables, this reduces to confirming 
   \[
   8 (\log q)^{3/4} + \edit{24} \log\log\edit{(126 q^3)} \edit{ + 4 \log\log q} + \log(\edit{(3 \times 10^{16})} 3^{\edit{8}} 100^{-2}) \leq 7 \log q + 107,
   \]
   for $q > 400,000$. We verified this inequality with  computer algebra software \textit{Maple}.  
\end{proof}

\begin{proof}[Proof of \cref{cor:non-explicit}]
    For any fixed $\epsilon \in (0,\frac{1}{2})$, this follows immediately from \cref{thm:main} with $\theta = \frac{1}{6}+\frac{\epsilon}{16}$ from \cref{thm:PY} and with $\frac{\epsilon}{2}$ from \cref{thm:Siegel}. 
\end{proof}

\begin{remark}
  As the proof illustrates, the choices for $c_2$ and $c_4$ in \cref{cor:explicit} are essentially optimal with respect to \cref{thm:main}, but the choices for $c_1$ and $c_3$ have quite a bit of latitude.   We chose $c_3 = 107$ because it is close to the value of $c_3$ in \eqref{eqn:ThornerZaman}. This allows one to more fairly compare the values of $c_1 = 10.0$ in \cref{cor:explicit} versus $c_1 = 54.2$ in \eqref{eqn:ThornerZaman}. Note  that by sufficiently inflating $c_3$ or by assuming $q$ is sufficiently large, the constant $c_1$ can be made arbitrarily close to $8\theta + 2\epsilon = 3$ with \cref{lem:ExplicitConvexity,thm:Bordignon}.  
\end{remark}

\section{Setup with the Selberg sieve}
\label{sec:SelbergSieve}

In this section, we set the scene for the proof of \cref{thm:main}. Henceforth, let $q > 400,000$ be a positive integer. Assume  that the function $\sL_q(s)$ given in \eqref{eqn:sL} has a simple real zero $s=\beta_1$ satisfying 
\begin{equation}
\beta_1 > 1 - \frac{1}{10 \log q}.
\label{eqn:beta_1}
\end{equation}
Let $\chi_1 \pmod{q}$ be the real quadratic character associated with this zero, so that $L(\beta_1,\chi_1) = 0$. Following the notation used in Graham's PhD thesis \cite[Section 9]{graham_applications_1977},  we put 
\begin{equation}
a(n) = \sum_{d \mid n} \chi_1(d),	
\label{eqn:a}
\end{equation}
for integers $n \geq 1$, so $a(n)$ is a non-negative multiplicative function. Note that we have the Dirichlet series equality
\[
\sum_{n=1}^{\infty}a(n)n^{-s} =\zeta(s)L(s,\chi_1),
\]
 for $\Re(s) >1$. Now, we prepare for Selberg's sieve following Graham's steps. Let $\chi \pmod{q}$ be an arbitrary Dirichlet character. Since $|a(n) \chi(n)| \ll_{\epsilon} n^{\epsilon}$ for any $\epsilon > 0$, the Dirichlet series 
\begin{equation} \label{eqn:F}
F(s,\chi)=\sum_{n=1}^{\infty}a(n)\chi(n)n^{-s} = \prod_{p} \sum_{k=0}^{\infty}a(p^{k})\chi(p^{k})p^{-ks},
\end{equation}
is absolutely convergent for $\Re(s)>1$. Let $R \geq 3$ be an arbitrary parameter. Let $\mu$ denote the M\"{o}bius function. Assume that the yet-to-be-specified weights $(\theta_d)_{d \geq 1}$ are subject to the conditions
\begin{equation}\label{restrictions on theta d}
  \begin{split} \theta_1&=1,\\    \theta_d&=0\,\,\text{if}\,\,\mu^2(d)=0\,\text{or}\,d>R.
  \end{split}
\end{equation}
Define the function $G(s,\chi)$ by 
\begin{align}\label{eqn:G}
G(s,\chi)=\sum_{d,e\leq R} \theta_d\theta_e \prod_{p\mid [d,e]} \frac{\sum_{k=1}^{\infty} a(p^k)\chi(p^k) p^{-ks}}{1+\sum_{k=1}^{\infty} a(p^k)\chi(p^k) p^{-ks}},
\end{align}
so that for $\Re(s) > 1$, 
\[
F(s,\chi)G(s,\chi)=\sum_{n=1}^{\infty}a(n)\chi(n) \Big(\sum_{d\mid n}\theta_d\Big)^2 n^{-s}.
\]
 
By a standard Selberg sieve technique, the weights $(\theta_d)_{d \geq 1}$ are chosen in \cite[Lemma 9.6]{graham_applications_1977} (c.f. Lemma \ref{lem:Graham} below) in order to estimate $G(1,\chi_0)$ optimally where $\chi_0$ is the principal character modulo $q$. To describe this choice, we must define two multiplicative functions $g$ and $h$. For primes $p$, put
\begin{align}\label{g h at primes}
  g(p)=\Big(\sum_{k=1}^{\infty} a(p^k) p^{-k}\Big)^{-1} \quad\text{and}\quad h(p)=g(p)+1\edit{,}
\end{align}
and for a positive integer $n$, set
\begin{align}\label{g and h}
  h(n)= \prod_{p^k||n}h(p)^k \quad\text{and}\quad g(n)=h(n) \prod_{p\mid n}\Big(1-\frac{1}{h(p)}\Big).
\end{align}
Note that $h$ is totally multiplicative and if $p\mid q$, then $a(p^k)=1$ so that $g(p)=p-1$, $h(p)=p$, $h(q)=q$, and 
\[g(q)=q\prod_{p\mid q} \Big(1-\frac{1}{p}\Big)=\varphi(q).\]
Also, as $a(n)\geq 0$, both $g$ and $h$ are nonnegative functions. Following Graham, we put
\begin{align}\label{optimal theta}
\theta_d=\frac{\mu(d)h(d)}{V(R)g(d)}\sum_{\substack{r\leq R/d\\(r,d)=1}} \frac{\mu^2(r)}{g(r)}, 
\end{align}
for $d\leq R$ where 
\begin{align}\label{VR def}V(R)=\sum_{\ell \leq R}\frac{\mu^2(\ell)}{g(\ell)}.
\end{align}
Note that if $\mu^{2}(d) = 0$ then $\theta_d=0$ (similarly if $d>R$, the empty sum over $r$ is treated as zero), so this choice satisfies the restrictions in \eqref{restrictions on theta d}. More importantly, this choice of $(\theta_{d})_{d \geq 1}$ gives the following upper bound for $G(1,\chi_0).$
\begin{lemma}[Lemma 9.6, \cite{graham_applications_1977}]\label{lem:Graham} Let $q > 400,000$ be an integer. Let $R \geq 200$ be arbitrary. Let $a(n)$ be the non-negative multiplicative function given by \eqref{eqn:a}. If the sequence $(\theta_d)_{d \geq 1}$ is chosen to satisfy \eqref{restrictions on theta d} and \eqref{optimal theta}, then $|\theta_d| \leq 1$ for all $d \geq 1$ and 
$$G(1,\chi_0)=\mathop{\sum\sum}_{\substack{d,e\leq R\\ (de,q)=1}} \frac{\theta_d\theta_e}{h([d,e])}\leq \frac{q}{\varphi(q)}\Big(\sum_{n\leq R}\frac{a(n)}{n}\Big)^{-1}, $$
where $\chi_0$ is the principal character modulo $q$. 
\end{lemma}

\begin{remark}
    \cref{lem:Graham} (and the following lemma as a consequence) is valid for any $R\geq 2$. As we will use both estimates for $R\geq 200$, we keep that assumption in the statement. 
\end{remark}
To utilize this estimate, we give a lower bound for the partial sum of the sequence $a(n)/n$. 
\begin{lemma}\label{lem:SelbergSieveMain} Let $q  > 400,000$.  Assume \cref{hypothesis_A} holds with constants $A \geq 1$ and $0 < \theta \leq \tfrac{1}{4}$. If $1 - \frac{1}{10 \log q} < \beta_{1} < 1$ is a real zero of $L(s,\chi_{1})$ where $\chi_1 \pmod{q}$ is a quadratic character, then 
    \[
        \sum_{n \leq R} \frac{a(n)}{n} \geq \frac{L(1,\chi_{1}) R^{1-\beta_{1}}}{(1-\beta_1)(2-\beta_1)} \Big( 1 - \frac{\edit{8}A e^{(\log q)^{3/4}} q^{\theta} }{R^{1/2}} \Big),
    \]
    provided $R \geq 200$. 
\end{lemma}

\begin{proof} 
Since $a(n)\geq 0$ and $1-\beta_{1} > 0$, it follows by monotonicity that
\[
    \sum_{n \le R}\frac{a(n)}{n} \ge \sum_{n \le R}\frac{a(n)}{n}\Big(\frac{n}{R}\Big)^{1-\beta_{1}} = R^{\beta_{1}-1}\sum_{n \le R}\frac{a(n)}{n^{\beta_{1}}}.
\]
By Mellin inversion, the innermost sum satisfies
\[
    \sum_{n \le R}\frac{a(n)}{n^{\beta_{1}}} \ge \sum_{n \le R}\frac{a(n)}{n^{\beta_{1}}}\Big(1-\frac{n}{R}\Big) 
     = \frac{1}{2\pi i}\int_{(2)}\zeta(s+\beta_{1})L(s+\beta_{1},\chi_{1})\frac{R^{s}}{s(s+1)}\,ds =: I.
\]
Shifting the line of integration to the line $\Re s = \frac{1}{2}-\beta_{1}$, we pickup a pole at $s = 1-\beta_{1}$ \,due to the zeta factor in the integrand. Note that the integrand does not have a pole at $s=0$ because of our assumption that $L(\beta_1,\chi_1)=0$. So, we have
\begin{equation*}
  \begin{split}
         I &= \frac{L(1,\chi_{1}) R^{1-\beta_{1}}}{(1-\beta_{1})(2-\beta_{1})}+\frac{1}{2\pi }\int_{-\infty}^{\infty} \frac{\zeta(\tfrac{1}{2}+it) L(\tfrac{1}{2}+it,\chi_1) R^{\frac{1}{2}-\beta_{1}+it}}{(\frac{1}{2}-\beta_{1}+it)(\frac{3}{2}-\beta_{1}+it)}\,dt\\&=:\frac{L(1,\chi_{1}) R^{1-\beta_{1}}}{(1-\beta_{1})(2-\beta_{1})}+\frac{1}{2\pi }I_1.
  \end{split}  
\end{equation*}
It remains to estimate $I_1$. By \cref{hypothesis_A}, \cref{lem:Phragmen}, noting that $\chi_1$ is not necessarily primitive, and the inequality $0.992 < 1 - \frac{1}{10 \log q} <\beta_1<1$ for $q > 400,000$, we have 
\begin{align*}
   \frac{|I_1|}{2\pi } & \leq \frac{A2^\theta q^{\theta} e^{(\log q)^{3/4}} R^{\frac{1}{2}-\beta_{1}}}{2\pi}\int_{-\infty}^{\infty}\frac{\Big|\zeta\Big(\frac{1}{2}+it\Big)\Big|(1+|t|)^{\theta}}{|\frac{1}{2}-\beta_{1}+it||\frac{3}{2}-\beta_{1}+it|}\,dt \\
   & \leq \frac{2^{1/4}Aq^{\theta} e^{(\log q)^{3/4}} R^{\frac{1}{2}-\beta_{1}}}{2\pi}\int_{-\infty}^{\infty}\frac{\Big|\zeta\Big(\frac{1}{2}+it\Big)\Big|(1+|t|)^{\theta}}{((0.492)^2+t^{2})^{\frac{1}{2}}((0.5)^2+t^{2})^{\frac{1}{2}}}\,dt. 
\end{align*}
By \edit{\cref{lem:ExplicitConvexity} for Riemann zeta} and the assumption $\theta \leq 1/4$, the above expression is at most
\[
\leq \frac{\edit{2^{1/4} \cdot 3} Aq^{\theta}e^{(\log q)^{3/4}} R^{\frac{1}{2}-\beta_{1}}}{\pi}\int_{0}^{\infty}\frac{(1+t)^{\edit{1/2}}}{(0.24+t^{2})^{\frac{1}{2}}(0.25+t^{2})^{\frac{1}{2}}}\,dt \leq \edit{5.5} A q^{\theta} e^{(\log q)^{3/4}} R^{1/2-\beta_1},
\]
since the integral is less than \edit{$4.8$} by direct computation and \edit{$4.8\cdot 2^{1/4} \cdot 3/\pi<5.5$}. Combining our calculations, we have established
\[
    \sum_{n \leq R} \frac{a(n)}{n} \geq \frac{L(1,\chi_{1}) R^{1-\beta_{1}}}{(1-\beta_1)(2-\beta_1)} \Big( 1 - \frac{\edit{5.5}A q^{\theta} e^{(\log q)^{3/4}} (1-\beta_1)(2-\beta_1) }{L(1,\chi_1) R^{1/2} } \Big). 
\]
The desired result follows from the lower bound in \cref{lem:RatioZeroL1} and the inequality $2-\beta_1 < 1 + \frac{1}{10 \log q} < 1.008$ for $q > 400,000$, noting that \edit{$5.5\times1.008/0.72\leq 7.7\leq 8.$}
\end{proof}
Moreover, we provide the following upper bound for $G(\frac{1}{2}+it,\chi)$  for $t \in \R$.  
\begin{lemma} \label{lem:G-estimate}
Keep the same notation and assumptions as \cref{lem:Graham}.  For $t \in \R$ and any Dirichlet character $\chi \pmod{q}$,
\[
|G(\tfrac{1}{2}+it,\chi)|\leq  12 (1-\beta_1)^2 (\log q)^4 R + \edit{5.4}\times10^7 A^2 q^{2\theta} e^{2(\log q)^{3/4}} \log^{12}(2eR),
\]
provided that $R \geq 200$.
\end{lemma}
\begin{remark}
   In \cite[p. 141]{graham_applications_1977}  Graham shows $|G(\frac{1}{2}+it,\chi)| \ll_{\delta} R^{1+\delta}$ for $\delta > 0$. Our estimate essentially scales this bound by an additional factor of $(1-\beta_1)^2$, so a more perilous Landau--Siegel zero yields stronger savings.  This savings allows us to choose the sifting level $R$ to be proportionally larger when detecting zeros in \cref{prop:ZeroDetector}. Ultimately, without this extra factor, the key exponent $8\theta+2\epsilon$ in \eqref{eqn:M} would inflate to $8\theta+4\epsilon$. This matters when taking $\epsilon = \frac{1}{2}$ via \cref{thm:Bordignon}, as we do for \cref{cor:explicit}. 
\end{remark}
  
\begin{proof} Let $s = \frac{1}{2} + it$. Since $a(p^k)=(1\ast \chi_1)(p^k)=1+\chi_{1}(p)+\dots + \chi_{1}(p^k)$, we observe that
\begin{align*}
G(s,\chi)=& \sum_{d \leq R} \sum_{e\leq R} \theta_d\theta_e \prod_{p\mid [d,e]}\Big(1-\Big(\sum_{k=0}^{\infty} a(p^k)\chi(p^k) p^{-ks}\Big)^{-1}\Big)\\
	= &\sum_{d \leq R} \sum_{e\leq R} \theta_d\theta_e \prod_{p\mid [d,e]} \Big(1-\Big( \Big(\sum_{k=0}^{\infty} \chi(p^k) p^{-ks}\Big)\Big( \sum_{j=0}^{\infty} \chi_{1}(p^j)\chi(p^j) p^{-js}\Big)\Big)^{-1}\Big)\\
	= &\sum_{d \leq R} \sum_{e\leq R} \theta_d\theta_e \prod_{p\mid [d,e]} \Big((1+\chi_{1}(p))\chi(p) p^{-s}-\chi_{1}(p)\chi(p^2) p^{-2s} \Big). 
\end{align*} 
Note that for $s=1/2+it$, 
\begin{align}\label{eq:bound for f at primes}
|(1+\chi_{1}(p))\chi(p) p^{-s}-\chi_{1}(p)\chi(p^2) p^{-2s}|\leq (1+\chi_1(p)) p^{-\frac{1}{2}} + p^{-1}.
\end{align}
For each prime $p$, we define $f(p) = 1+\chi_1(p) + p^{-1/2} >  0$, extend the definition of $f$  multiplicatively to all square-free integers and put $f(n)=0$ if $n$ is not square-free. We define its Dirichlet series by 
\[
F(s) = \sum_{n=1}^{\infty} f(n) n^{-s}.
\]
Now, since $|\theta_d| \leq 1$ by \eqref{restrictions on theta d} and \cref{lem:Graham},  \eqref{eq:bound for f at primes} implies 
\begin{align*}
    |G(s,\chi)|\leq \sum_{d \leq R} \sum_{e\leq R}  f([d,e]) [d,e]^{-1/2}\leq \Big(\sum_{d\leq R} f(d) d^{-1/2} \Big)^2.
\end{align*}
As $f$ is non-negative, it follows by standard Mellin inversion that
\begin{equation}
\sum_{d\leq R}  f(d) d^{-1/2} \leq 2 \sum_{d \leq 2R}  \frac{f(d)}{d^{1/2}} \Big(1 - \frac{d}{2R}\Big) = \frac{2}{2\pi i} \int_{(2)} F(s + \tfrac{1}{2}) \frac{ 2^sR^s}{s(s+1)}  \,ds. 
\label{eqn:SelbergEstimate}
\end{equation}
To shift the contour, we must verify some analytic properties of $F(s)$. By comparing Euler products, one can verify by direct calculation that 
\begin{equation} \label{eqn:CompareEulerproducts}
F(s) = \zeta(s) L(s,\chi_1)  B(s), 
\end{equation}
for $\Re(s) > 1$, where
\begin{align*}
B(s) 
&	=  \prod_p \Big( 1 +  p^{-s-1/2} + b_2(p) p^{-2s} +  b_3(p) p^{-3s}  \Big), \\
b_2(p) & = -1 - \chi_1(p) - \chi_1(p)^2 - \frac{1+\chi_1(p)}{p^{1/2}},  \\
b_3(p) & = \chi_1(p) + \chi_1(p)^2 + \frac{\chi_1(p)}{p^{1/2}}.
\end{align*}
Since $b_2(p), b_3(p)  \ll 1$, it follows that $B(s)$ converges absolutely for $\Re(s) > 1/2$. As $(s-1)\zeta(s) L(s,\chi_1)$ is entire, the identity \eqref{eqn:CompareEulerproducts} therefore implies that $F(s)$ has analytic continuation for $\Re(s) > 1/2$ with a simple pole at $s=1$. 

Thus, we may shift the contour for the integral in \eqref{eqn:SelbergEstimate}  over to the line $\Re(s) = \delta$ for any fixed $\delta \in (0,\tfrac{1}{2})$ and pick up the simple pole at $s=1/2$. This gives
\[
\sum_{d\leq R}  f(d) d^{-1/2} \leq \frac{8\sqrt{2}}{3} L(1,\chi_1) B(1)   R^{1/2} + \frac{1}{\pi} \int_{-\infty}^{\infty} \frac{F(\tfrac{1}{2} +\delta + it)   (2R)^{\delta+it}}{(\delta+it)(1+\delta+it)} dt. 
\]
To crudely estimate $B(1)$, notice that $b_2(p) \leq 0$ and $b_3(p) \leq 3$ for all primes $p$ so  
\begin{align*}
B(1) \leq& \prod_p \Big( 1 + p^{-\frac{3}{2}} + 3p^{-3} \Big) =\prod_{p\le 40000} \Big( 1 + p^{-\frac{3}{2}} + 3p^{-3} \Big) \prod_{p>40000} \Big( 1 + p^{-\frac{3}{2}} + 3p^{-3} \Big) 
\\ \le& 3.15 \prod_{p>40000} \Big( 1 + p^{-\frac{3}{2}} + 3p^{-3} \Big), 
\end{align*}
by direct computation. Also 
\[
\prod_{p>40000} \Big( 1 + p^{-\frac{3}{2}} + 3p^{-3} \Big) \le \Big(1+\sum_{n=40000}^\infty \frac{1}{n^{3/2
}}\Big)\Big(1+\sum_{n=40000}^\infty \frac{3^{\omega(n)}}{n^{3}}\Big) \leq 1.1,  
\]
so that $B(1)\leq 3.15\times 1.1\leq 3.5.$ Similarly, we have $|b_2(p)| \leq 3 + 2p^{-1/2}$ and $|b_3(p)| \leq 3$. So for $t \in \R$, using the standard estimate $\zeta(1+\sigma) \leq 1 + \sigma^{-1}$ for $\sigma > 0$, we have 
\begin{align*}
|B(\tfrac{1}{2}+\delta + it)| 
&  \leq \prod_p \Big( 1 + p^{-1-\delta } + 3p^{-1-2\delta} + 5 p^{-\frac{3}{2}} \Big) \\
&  \leq  \zeta(1+\delta) \zeta(1+2\delta)^3  \zeta(\tfrac{3}{2})^5 \\
& \leq 3^5 \big( 1 + \delta^{-1} \big)^4.
\end{align*}
It follows from \edit{\cref{lem:ExplicitConvexity} for Riemann zeta}  and \cref{lem:Phragmen} with $\eta=\delta$ and $\sigma = \frac{1}{2}+\delta$ that
\[
|\zeta(\tfrac{1}{2}+\delta+it)| \leq  \edit{3} (2 (1+|t|) )^{\edit{1/4}}  \Big(1+\delta^{-1} \Big)^{2\delta}.
\] 
By \cref{hypothesis_A} and \cref{lem:Phragmen} with $\eta=\delta$ and $\sigma = \frac{1}{2}+\delta$, we also have
\[
|L(\tfrac{1}{2}+\delta+it,\chi_1)| \leq \Big( A (2q (1+|t|) )^{\theta}  \Big) \Big(1+\delta^{-1} \Big)^{2\delta} e^{(\log q)^{3/4}}.
\]
Overall, as $\edit{\frac{1}{4}}+\theta \leq \frac{1}{2}$ and $3^5  \cdot \edit{3 \cdot 2^{1/4+\theta}} \leq \edit{1031}$, all of these estimates together give
\[
|F(\tfrac{1}{2}+\delta+it)| \leq \edit{1031} A  e^{(\log q)^{3/4}} q^{\theta} (1+|t|)^{1/2} (1+\tfrac{1}{\delta})^{4+4\delta}.
\]
Inserting this estimate into the integral, we obtain that
\begin{align*}
\Big| \frac{1}{\pi} \int_{-\infty}^{\infty} \frac{F(\tfrac{1}{2} +\delta + it)   (2R)^{\delta+it}}{(\delta+it)(1+\delta+it)} dt \Big|	
& \leq \frac{\edit{1031}}{\pi} A e^{(\log q)^{3/4}} q^{\theta} (2R)^{\delta} (1+\tfrac{1}{\delta})^{4+4\delta} \int_{-\infty}^{\infty} \frac{(1+|t|)^{1/2}}{\sqrt{(\delta^2+t^2)(1+t^2)}} dt  \\
& \leq \frac{\edit{1031}}{\pi} A e^{(\log q)^{3/4}} q^{\theta} (2R)^{\delta} (1+\tfrac{1}{\delta})^{5+4\delta} \int_{-\infty}^{\infty} \frac{(1+|t|)^{1/2}}{1+t^2} dt  \\
& \leq \frac{\edit{5980}}{\pi} A e^{(\log q)^{3/4}} q^{\theta} (2R)^{\delta} (1+\tfrac{1}{\delta})^{5+4\delta},
\end{align*}
where, in the second step, we used that $\delta^2+t^2\geq\delta^2 (1+t^2)$ for $\delta \in (0,1/2)$ and that the integral on the right is at most 5.8. 

Finally, taking $\delta = \frac{1}{\log(2R)} \leq \frac{1}{\log (400)} < \frac{1}{4}$, we obtain
\[
\Big| \frac{1}{\pi} \int_{-\infty}^{\infty} \frac{F(\tfrac{1}{2} +\delta + it)   (2R)^{\delta+it}}{(\delta+it)(1+\delta+it)} dt \Big| \leq \edit{5175} A e^{(\log q)^{3/4}} q^{\theta} \log^6(2eR), 
\]
since $\edit{5980}\cdot e/\pi  < \edit{5175}$. Combining this with our previous estimates and the standard inequality $(x+y)^2 \leq 2x^2+2y^2$ for $x,y \geq 0$, gives
\begin{align*}
|G(\tfrac{1}{2}+it,\chi)| 
& \leq \Big( \frac{8\sqrt{2}}{3} L(1,\chi_1)   \cdot 3.5 \cdot  R^{1/2} + \edit{5175} A e^{(\log q)^{3/4}} q^{\theta} \log^6(2eR) \Big)^2  \\
& \leq 350 L(1,\chi_1)^2 R + \edit{54}000000 A^2 e^{2(\log q)^{3/4}} q^{2\theta} \log^{12}(2eR) \\
& \leq 12 (1-\beta_1)^2 (\log q)^4 R + \edit{54}000000 A^2 e^{2(\log q)^{3/4}} q^{2\theta} \log^{12}(2eR),
\end{align*}
as required. The last inequality follows from \cref{lem:RatioZeroL1}. 
\end{proof}


\section{Proof of \cref{thm:main}} \label{sec:Proof}
Let $q > 400,000$ be an integer and let $T \geq 4$. Assume the following.
\begin{itemize}
    \item The function $\sL_q(s)$ in \eqref{eqn:sL} has a real zero at $s = \beta_1 > 1 - \frac{1}{10 \log q}$ associated to real quadratic character $\chi_1 \pmod{q}$.
    \item \cref{hypothesis_A} holds with fixed $A \geq 1$ and $0 < \theta \leq \frac{1}{4}$.
    \item \cref{hypothesis_B} holds with fixed $B \geq 1$ and $0 < \epsilon \leq \frac{1}{2}$.
\end{itemize} 
Choose the sequence $(\theta_d)_{d \geq 1}$ according to \eqref{restrictions on theta d} and \eqref{optimal theta} with parameter
\begin{equation}
    R =  \edit{256} A^2 q^{2 \theta} e^{2(\log q)^{3/4}} (1-\beta_1)^{-2}. 
    \label{eqn:R-choice}
\end{equation}
Also put
\begin{equation}\label{eqn:N-choice}
  N =  \edit{(3 \times 10^{16})} A^{\edit{8}} e^{8 (\log q)^{3/4}  } (\edit{\log(14A^2 q^3)} )^{24}  q^{8\theta} T^{4\theta}  (1-\beta_1)^{-2}.
\end{equation}
Note that $R \geq 200$ and $N \geq 10^{25}$ since $\beta_1 > 1 - \frac{1}{10 \log q}$ \edit{ and $q > 400000$}. Moreover, $N \leq M$ by \cref{hypothesis_B} and \eqref{eqn:M}. 
Before introducing another zero of $\sL_q(s)$, we establish the key proposition for detecting zeros. 
\begin{proposition}\label{prop:ZeroDetector} Keep the notation above. 
Let $\chi$ be a Dirichlet character modulo $q$. If $\rho = \beta+i\gamma$ is a non-trivial zero of $L(s,\chi)$ or $L(s,\chi\chi_1)$ satisfying $\beta > 1/2$, then  
\[
    \Big|  \sum_{n \leq N} a(n)\chi(n) \Big( \sum_{d \mid n} \theta_d \Big)^2  n^{-\rho}  \Big( 1 -\frac{n}{N} \Big) \Big| \leq  \Big( \frac{2-\beta_1}{1-\beta}   + \frac{1}{\beta-\frac{1}{2}} \Big) (1-\beta_1) N^{1-\beta}. 
\]
\end{proposition}
\begin{proof}  
   Recalling the definitions of $a(n)$,  $F(s,\chi)$, and $G(s,\chi)$ from \eqref{eqn:a}, \eqref{eqn:F}, and \eqref{eqn:G}  respectively,  we have $F(s,\chi) = L(s,\chi) L(s,\chi\chi_1)$. By Mellin inversion, we have
    \[
        \sum_{n \leq N} a(n)\chi(n) \Big( \sum_{d \mid n} \theta_d \Big)^2  n^{-\rho}  \Big( 1 -\frac{n}{N} \Big)=\frac{1}{2\pi i}\int_{(2)}F(s+\rho,\chi)G(s+\rho,\chi) \frac{N^{s}}{s(s+1)}\,ds=:I.
    \]
    Shifting the line of integration to the line $\Re(s) = \frac{1}{2}-\beta$, our assumption that $L(\rho,\chi) = 0$ or $L(\rho,\chi\chi_1) = 0$ implies that we only pick up a possible simple pole at $s=1-\rho$ if $\chi = \chi_0$ or $\chi_1$. In either case, since
    \[
    \mathop{\mathrm{Res}}_{s = 1}L(s,\chi_{0}) = \prod_{p \mid q}(1-p^{-1})=\frac{\varphi(q)}{q}
    \quad \text{ and  } \quad G(1,\chi_{1})=G(1,\chi_0),
    \]
    as $\chi_{1}(n)a(n)=a(n)\chi_0(n)$, \cref{lem:Graham} implies
    \[
        |I| \leq \frac{L(1,\chi_1)N^{1-\beta}}{|(1-\rho)(2-\rho)|}\Big(\sum_{n \le R}\frac{a(n)}{n}\Big)^{-1}+|J|,
    \]
    where
    \[
        J := \frac{1}{2\pi} \int_{-\infty}^{\infty}F(\tfrac{1}{2}+i(t+\gamma), \chi)G(\tfrac{1}{2}+i(t+\gamma),\chi)\frac{N^{\frac{1}{2}-\beta+it} }{(\frac{1}{2}-\beta+it)(\frac{3}{2}-\beta+it)}\,dt.
    \]
  By \cref{lem:SelbergSieveMain}, we have the bound for the term due to the residue
    \begin{align*}
        \frac{	L(1,\chi_1)N^{1-\beta}}{|(1-\rho)(2-\rho)|}\Big(\sum_{n \le R}\frac{a(n)}{n}\Big)^{-1} &\leq \frac{(1-\beta_1)(2-\beta_1) }{|(1-\rho)(2-\rho)|} \frac{ N^{1-\beta}}{R^{1-\beta_1}}  \Big( 1 - \frac{\edit{8}A q^{\theta}e^{(\log q)^{3/4}}}{R^{1/2} } \Big)^{-1} \\
        & \leq  \frac{(1-\beta_1)(2-\beta_1) }{|(1-\rho)(2-\rho)|}  N^{1-\beta} \Big(1 - \tfrac{1}{2}(1-\beta_1) \Big)^{-1},
    \end{align*}
    where the last inequality follows from the choice of $R$ in \eqref{eqn:R-choice} and the trivial bound $R^{1-\beta_1} \geq 1$. Next, we estimate the integral $J$. Recalling $F(s,\chi) = L(s,\chi) L(s,\chi\chi_1)$, by \cref{lem:Phragmen}, we have
    \begin{align*}
        |J| & \leq \frac{1}{2\pi} \int_{-\infty}^{\infty} |F(\tfrac{1}{2}+i(t+\gamma), \chi) \cdot G(\tfrac{1}{2}+i(t+\gamma),\chi)|\frac{N^{\frac{1}{2}-\beta} }{|\frac{1}{2}-\beta+it||\frac{3}{2}-\beta+it|}\,dt \\
        & \leq \frac{N^{1/2-\beta}}{2\pi} \int_{-\infty}^{\infty} ( A (2q (1+|t+\gamma|) )^{\theta} )^2 e^{2 (\log q)^{3/4} }   \frac{ |G(\frac{1}{2}+i(t+\gamma),\chi)|}{((\frac{1}{2}-\beta)^2 + t^2 )^{1/2}  ((\frac{3}{2}-\beta)^2 + t^2 )^{1/2} }  \,dt \\
        & \edit{\leq } \edit{\frac{2^{2\theta}}{2\pi} A^2} q^{2\theta}  e^{2 (\log q)^{3/4} } N^{1/2-\beta} \int_{-\infty}^{\infty} \frac{ (1+|t+\gamma|)^{2\theta} |G(\frac{1}{2}+i(t+\gamma),\chi)|}{((\frac{1}{2}-\beta)^2 + t^2 )^{1/2}  ((\frac{3}{2}-\beta)^2 + t^2 )^{1/2} }\,dt. 
    \end{align*}
   By \cref{lem:G-estimate} and the choice of $R$ in \eqref{eqn:R-choice}, we obtain
    \begin{align*}
        |G(\tfrac{1}{2}+i(t+\gamma),\chi)| 
        & \leq  12 (1-\beta_1)^2 (\log q)^4 R + \edit{5.4}\times10^7 A^2 q^{2\theta} e^{2(\log q)^{3/4}} (\log  2eR)^{12} \\
        & \leq  \big(  \edit{3072} (\log q)^4 + \edit{5.4} \times 10^{7} (\edit{ \log (14A^2q^3)})^{12} \big) A^2 q^{2\theta}  e^{2(\log q)^{3/4}}\\
        & \leq  \edit{5.5 \times 10^{7}} A^{\edit{2}} q^{2\theta} e^{2(\log q)^{3/4}} (\edit{\log (14A^2q^3)})^{12},
    \end{align*}
    because \cref{thm:Bordignon}, $A \geq 1$, and \edit{$e^{2(\log q)^{3/4}}(\log q)^4 \leq 100 q^{3/2}$ for $q > 400000$} imply the crude bound 
    \begin{align*}
    \log(2eR) = \log\big( \edit{512} e A^2 q^{2\theta} \edit{e^{2(\log q)^{3/4}}} (1-\beta_1)^{-2} \big) 
    & \leq \log\big(  \tfrac{\edit{512}e}{\edit{10000}} A^2 q^{3/2} \edit{e^{2(\log q)^{3/4}}}  (\log q)^4 \big) \\
    & \leq \edit{ \log(\tfrac{512e}{10000} A^2 \cdot 100 q^{3} ) } \\
    & \leq \edit{\log (14A^2q^3)}. 
    \end{align*}
    Using the weak bound $(\frac{3}{2}-\beta)^2 + t^2 \geq (\frac{1}{2}-\beta)^{2}+t^2$ with our estimate for $|G(\frac{1}{2}+i(t+\gamma),\chi)|$, we conclude that 
    \[
        |J| \leq \edit{\dfrac{2^{2\theta}}{2\pi}} (\edit{5.5 \times 10^{7}} )A^{\edit{4}}  q^{4\theta}  e^{4 (\log q)^{\frac{3}{4}} }(\edit{\log (14A^2q^3)})^{12} N^{1/2-\beta} \int_{-\infty}^{\infty} \frac{ (1+|t+\gamma|)^{2\theta}}{(\frac{1}{2}-\beta)^{2}+t^2}  \,dt.
    \]
    The elementary bound
    \[
        \int_{|t| \le 2|\gamma|} \frac{ (1+|t+\gamma|)^{2\theta}}{(\frac{1}{2}-\beta)^{2}+t^2} \,dt \leq    
        \int_{-\infty}^{\infty} \frac{2(1+|\gamma|)^{2\theta}}{(\frac{1}{2}-\beta)^{2}+t^2} \,dt = \frac{2\pi(1+|\gamma|)^{2\theta}}{\beta-\frac{1}{2}},
    \]
    and the smaller estimate
    \[
        \int_{|t| > 2|\gamma|} \frac{ (1+|t+\gamma|)^{2\theta}}{(\frac{1}{2}-\beta)^{2}+t^2} \,dt \leq 2 \int_{2|\gamma|}^{\infty} \frac{(1+t)^{2\theta}}{(\frac{1}{2}-\beta)^{2}+t^2} \,dt \leq 2 \int_{0}^{\infty} \frac{1+t^{\frac{1}{4}}}{(\frac{1}{2}-\beta)^{2}+t^2} \,dt \leq \frac{2\pi}{\beta-\frac{1}{2}},
    \]
    together imply
    \[
        \int_{-\infty}^{\infty} \frac{ (1+|t+\gamma|)^{2\theta}}{(\frac{1}{2}-\beta)^{2}+t^2}  \,du \leq \frac{\edit{4\pi}(1+|\gamma|)^{2\theta}}{\beta - \frac{1}{2}}.
    \]
    Therefore we arrive at
    \[
        |J| \leq (\edit{1.6 \times 10^{8}}) A^{\edit{4}} q^{4\theta}e^{4 (\log q)^{\frac{3}{4}}  } (\edit{\log (14A^2q^3)})^{12} (1+|\gamma|)^{2\theta} \frac{N^{1/2-\beta}}{\beta-\frac{1}{2}}.
    \]
    Combining our estimate for the residue term with this estimate, gives the following upper bound for the sum:
    \[
        \frac{(1-\beta_1)(2-\beta_1) }{|(1-\rho)(2-\rho)|}  N^{1-\beta} \big(1 - \tfrac{1}{2}(1-\beta_1) \big)^{-1}+ (\edit{1.6 \times 10^{8}}) A^{\edit{4}} q^{4\theta}e^{4 (\log q)^{\frac{3}{4}}  } (\edit{\log (14A^2q^3)})^{12} (1+|\gamma|)^{2\theta} \frac{N^{1/2-\beta}}{\beta-\frac{1}{2}}.
    \]
    The desired bound now follows from the choice of $N$ in \eqref{eqn:N-choice} as well as the inequalities $2-\beta_1 \leq |2-\rho|$, $|1-\rho| \geq 1-\beta$, $1+|\gamma| \leq \frac{5}{4} T$, $(1-x)^{-1} \leq 1+2x$ for $0 < x < 0.1$.
\end{proof}

Now, we may complete the proof of \cref{thm:main}. Let  $\chi \pmod{q}$ be a Dirichlet character. Assume $\rho=\beta+i\gamma$ is a nontrivial zero of $
   L(s,\chi)$ such that $\rho \neq \beta_1$, $\beta>1/2$,  and $|\gamma|\le T$. Since $N \leq M$ by \cref{hypothesis_B} and \eqref{eqn:M}, and the righthand side of \eqref{eqn:main} is increasing with respect to $M$, it suffices to prove  
   \begin{equation} \label{eqn:proof-N-inequality}
       \beta < 1 - \frac{\log(\frac{\theta}{4(1-\beta_1) \log N}\big)}{\log N}. 
   \end{equation}
    We claim that we may assume $\beta > \frac{1}{2} + \frac{5\theta}{\log N}$. Indeed, our choice of $N$ in \eqref{eqn:N-choice} implies 
   \[
    \frac{\theta}{4(1-\beta_1)} < N^{\frac{1}{2} - \frac{5\theta}{\log N}} \log N,
    \]
    in which case,
    \[
        \frac{1}{2} + \frac{5\theta}{\log N} < 1 - \frac{\log\big(\frac{\theta}{4(1-\beta_1) \log N}\big)}{\log N}.
    \]
  Thus, if $\beta \leq \frac{1}{2} + \frac{5\theta}{\log N}$, \eqref{eqn:proof-N-inequality} is automatically satisfied. This proves the claim. We begin the main argument in earnest. Define 
    \begin{align*}
        S = \sum_{n \leq N}  \Big(\sum_{d|n}\chi_{1}(d) \Big) \Big( \sum_{d \mid n} \theta_d \Big)^2 \chi(n) n^{-\rho} \Big(1 - \frac{n}{N}\Big). 
    \end{align*}
Isolating the term with $n=1$ and taking the modulus, we have by reverse triangle inequality and the non-negativity of $\sum_{d \mid n} \chi_1(d)$ that
\begin{align*}
        |S| &= \Big|1 - \frac{1}{N} +\sum_{2 \leq n \leq N} \Big(\sum_{d|n}\chi_{1}(d) \Big) \Big( \sum_{d \mid n} \theta_d \Big)^2 \chi(n) n^{-\rho} \Big(1-\frac{n}{N}\Big)\Big|\\
        &\geq 1 - \frac{1}{N} - \sum_{2 \leq n \leq N} \Big(\sum_{d|n}\chi_{1}(d) \Big) \Big( \sum_{d \mid n} \theta_d \Big)^2 \chi_0(n) n^{-\beta} \Big(1-\frac{n}{N}\Big) \numberthis\label{s and s1 sum} \\
        & =: 1 - \frac{1}{N} - S_0.
    \end{align*}
    As $N \geq 10^{25}$, it follows that
    \begin{equation}
        \label{eqn:Proof-TwoSums}
   	1 - 10^{-25} \leq S_0 + |S|. 
    \end{equation}
    We shall estimate both $S_0$ and $|S|$ using \cref{prop:ZeroDetector}. First, we estimate $S_0$. Since $n^{-\beta} \leq n^{-\beta_1} N^{\beta_1-\beta}$ for $n \leq N$, we have that 
    \begin{align*}
    S_0 & \leq  N^{\beta_1-\beta} \sum_{2 \leq n \leq N} \Big(\sum_{d|n}\chi_{1}(d) \Big) \Big( \sum_{d \mid n} \theta_d \Big)^2  \chi_0(n) n^{-\beta_1} \Big(1-\frac{n}{N}\Big)  \\ 
        & \leq - N^{\beta_1-\beta} \big(1 - N^{-1} \big) + N^{\beta_1-\beta} \sum_{n \leq N} \Big(\sum_{d|n}\chi_{1}(d) \Big) \Big( \sum_{d \mid n} \theta_d \Big)^2 \chi_0(n) n^{-\beta_1} \Big(1-\frac{n}{N}\Big). 
    \end{align*}
 As $L(\beta_1,\chi_1) = 0$ by our assumption,  \cref{prop:ZeroDetector} implies
    \begin{align*}
 &    \Big|   \sum_{n \leq N} \Big(\sum_{d|n}\chi_{1}(d) \Big) \Big( \sum_{d \mid n} \theta_d \Big)^2 \chi_0(n) n^{-\beta_1} \Big(1-\frac{n}{N}\Big) \Big| \leq \Big(2-\beta_1 + \frac{1-\beta_1}{\beta_1-\frac{1}{2}}\Big)  N^{1-\beta_1}. 
    \end{align*}
 Since $\beta_1 -\frac{1}{2} > \frac{1}{2} - \frac{1}{10 \log q} > 0.492$ for $q > 400,000$ and $N^{1-\beta_1}<N^{1-\beta}$, this implies that
	\begin{align*}
	S_0 & \leq -N^{\beta_1-\beta} (1-N^{-1}) +  N^{1-\beta} (1 + 3.1(1-\beta_1) )   \\
	& =  \big( 1 - N^{-1+\beta_1} + N^{-2+\beta_1} + 3.1(1-\beta_1)  \big)N^{1-\beta}\\
	& \leq   \big((1-\beta_1) \log N + N^{-1} + 3.1 (1-\beta_1) \big) N^{1-\beta} \\
& \leq  \tfrac{1}{2\theta} (1-\beta_1)   N^{1-\beta} \log N,
 \end{align*}
	because $1- e^{-x} \leq x$ for $x > 0$, $0 < \theta \leq \frac{1}{4}$, and $N \geq 10^{25} (1-\beta_1)^{-1}$ by \eqref{eqn:N-choice}. Second, we estimate $|S|$. Since $L(\rho,\chi) = 0$ by assumption, \cref{prop:ZeroDetector} implies
    \begin{align*}
    |S| & \leq  \Big( \frac{2-\beta_1}{1-\beta}  + \frac{1}{\beta-\frac{1}{2}} \Big) (1-\beta_1) N^{1-\beta}. 
    \end{align*}
    From McCurley's zero-free region \eqref{eqn:ZFR-McCurley}, we have that $1-\beta \geq (10 \log qT)^{-1} \geq 4\theta (10 \log N)^{-1}$. Also, $2-\beta_1 < 1 + \frac{1}{10 \log q} < 1.008$ for $q > 400,000$ and $\beta > \frac{1}{2} + \frac{5\theta}{\log N}$ by assumption. Combining our observations gives
    \[
    |S| \leq   (\tfrac{10.08}{4\theta} + \tfrac{1}{5\theta}) (1-\beta_1) N^{1-\beta} \log N \leq \tfrac{55}{20\theta} (1-\beta_1)   N^{1-\beta} \log N. 
    \]
    Inserting our estimates for $S_0$ and $|S|$ into \eqref{eqn:Proof-TwoSums}, we find that 
	\begin{align*}
	1 - 10^{-25
 } \leq \tfrac{65}{20\theta} (1-\beta_1) N^{1-\beta} \log N. 
	\end{align*}
	By rearranging and taking logarithms, we have proved \eqref{eqn:proof-N-inequality} and hence \cref{thm:main}. \qed 

\section*{Acknowledgements} This project was initiated in July 2023 as part of the Inclusive Pathways in Explicit Number Theory  summer school, where we have benefited from conversations with many colleagues. We are thankful to BIRS, UBC Okanagan, UBC Vancouver, and PIMS for financial support and for providing conducive research environments. We are also grateful to  Alia Hamieh, Ghaith Hiary, Habiba Kadiri, Allysa Lumley, Greg Martin, and Nathan Ng for organizing the summer school and facilitating this collaboration. We also thank Daniel Johnston, \edit{Rick Lu}, Jesse Thorner, \edit{Kristopher Zhao}, and the anonymous referees for helpful comments on an earlier version of this manuscript. AZ was partially supported by NSERC grant RGPIN-2022-04982. 

\bibliographystyle{alpha}
\bibliography{referencesfinal}

\end{document}